\documentclass[10pt,a4paper]{preprint}
\usepackage{amsfonts,amsmath,amssymb, times}
\usepackage{mhequ}
\usepackage{color}
\newtheorem{theorem}{Theorem}
\newtheorem{lemma}{Lemma}

\numberwithin{equation}{section}
\def\sectionmark#1{\markboth{}{}}

\begin{document}

\def\ssm{\smallsetminus}
\newcommand{\tmname}[1]{\textsc{#1}}
\newcommand{\tmop}[1]{\operatorname{#1}}
\newcommand{\tmsamp}[1]{\textsf{#1}}
\newenvironment{enumerateroman}{\begin{enumerate}[i.]}{\end{enumerate}}
\newenvironment{enumerateromancap}{\begin{enumerate}[I.]}{\end{enumerate}}

\newcounter{problemnr}
\setcounter{problemnr}{0}
\newenvironment{problem}{\medskip
  \refstepcounter{problemnr}\small{\bf\noindent Problem~\arabic{problemnr}\ }}{\normalsize}
\newenvironment{enumeratealphacap}{\begin{enumerate}[A.]}{\end{enumerate}}
\newcommand{\tmmathbf}[1]{\boldsymbol{#1}}

\def\paral{/\kern-0.55ex/}
\def\parals_#1{/\kern-0.55ex/_{\!#1}}
\def\bparals_#1{\breve{/\kern-0.55ex/_{\!#1}}}
\def\n#1{|\kern-0.24em|\kern-0.24em|#1|\kern-0.24em|\kern-0.24em|}
\newenvironment{proof}{
 \noindent\textbf{Proof}\ }{\hspace*{\fill}
  \begin{math}\Box\end{math}\medskip}

\newcommand{\A}{{\bf \mathcal A}}
\newcommand{\B}{{\bf \mathcal B}}
\def\C{\mathbb C}
\newcommand{\D}{{\rm I \! D}}
\newcommand{\dom}{{\mathcal D}om}
\newcommand{\pathR}{{\mathcal{\rm I\!R}}}
\newcommand{\Nabla}{{\bf \nabla}}
\newcommand{\E}{{\mathbf E}}
\newcommand{\Epsilon}{{\mathcal E}}
\newcommand{\F}{{\mathcal F}}
\newcommand{\G}{{\mathcal G}}
\def\g{{\mathfrak g}}
\newcommand{\HH}{{\mathcal H}}
\def\h{{\mathfrak h}}
\def\k{{\mathfrak k}}
\newcommand{\I}{{\mathcal I}}
\def\LL{{\mathbb L}}
\def\law{\mathop{\mathrm{ Law}}}
\def\m{{\mathfrak m}}
\newcommand{\K}{{\mathcal K}}
\newcommand{\p}{{\mathfrak p}}
\def\P{\mathbb{P}}
\newcommand{\R}{{\mathbb R}}
\newcommand{\Rc}{{\mathcal R}}
\def\T{{\mathcal T}}
\def\M{{\mathcal M}}
\def\N{{\mathcal N}}
\newcommand{\pnabla}{{\nabla\!\!\!\!\!\!\nabla}}
\def\X{{\mathbb X}}
\def\Y{{\mathbb Y}}
\def\L{{\mathcal L}}
\def\1{{\mathbf 1}}
\def\half{{ \frac{1}{2} }}
\def\vol{{\mathop {\rm vol}}}
\def\euc{{\mathop {\rm eul}}}

\newcommand{\term}{{1}}
\newcommand{\tm}{{t}}
\newcommand{\stm}{{s}}
\newcommand{\pole}{{y_0}}

\def\ad{{\mathop {\rm ad}}}
\def\Conj{{\mathop {\rm Ad}}}
\def\Ad{{\mathop {\rm Ad}}}
\newcommand{\const}{\rm {const.}}
\newcommand{\eg}{\textit{e.g. }}
\newcommand{\as}{\textit{a.s. }}
\newcommand{\ie}{\textit{i.e. }}
\def\s.t.{\mathop {\rm s.t.}}
\def\esssup{\mathop{\rm ess\; sup}}
\def\Ric{{\mathop{\rm {Ric}}}}
\def\ric{{\mathop{\rm ric}}}
\def\div{\mathop{\rm div}}
\def\ker{\mathop{\rm ker}}
\def\Hess{\mathop{\rm Hess}}
\def\Image{\mathop{\rm Image}}
\def\Dom{\mathop{\rm Dom}}
\def\id{\mathop {\rm id}}
\def\Image{\mathop{\rm Image}}
\def\Cyl{\mathop {\rm Cyl}}
\def\Conj{\mathop {\rm Conj}}
\def\Span{\mathop {\rm Span}}
\def\trace{\mathop{\rm trace}}
\def\ev{\mathop {\rm ev}}
\def\supp{{\mathrm supp}}
\def\Ent{\mathop {\rm Ent}}
\def\tr{\mathop {\rm tr}}
\def\graph{\mathop {\rm graph}}
\def\loc{\mathop{\rm loc}}
\def\so{{\mathfrak {so}}}
\def\su{{\mathfrak {su}}}
\def\u{{\mathfrak {u}}}
\def\o{{\mathfrak {o}}}
\def\pp{{\mathfrak p}}
\def\gl{{\mathfrak gl}}
\def\hol{{\mathfrak hol}}
\def\z{{\mathfrak z}}
\def\t{{\mathfrak t}}
\def\<{\langle}
\def\>{\rangle}
\def\span{{\mathop{\rm span}}}
\def\diam{\mathrm {diam}}
\def\inj{\mathrm {inj}}
\def\Lip{\mathrm {Lip}}
\def\Iso{\mathrm {Iso}}
\def\Osc{\mathop{\rm Osc}}
\renewcommand{\thefootnote}{}
\def\supp{\mathrm {Supp}}
\def\V{\mathbb V}
\def\vol{{\mathop {\rm vol}}}
\def\cut{{\mathop {\rm Cut}}}
\def\Lip{{\mathop {\rm Lip}}}
\def\Cyl{{\mathop {\rm Cyl}}}

\def\paral{/\kern-0.55ex/}
\def\parals_#1{/\kern-0.55ex/_{\!#1}}
\def\bparals_#1{\breve{/\kern-0.55ex/_{\!#1}}}
\def\n#1{|\kern-0.24em|\kern-0.24em|#1|\kern-0.24em|\kern-0.24em|}
\def\f{\frac}
\title{On the Semi-Classical Brownian Bridge Measure}
\author{Xue-Mei Li }
\institute{Mathematics Institute, The University of Warwick, Coventry CV4 7AL, U.K.}
\maketitle

\begin{abstract}
We prove an integration by parts formula for the probability measure induced by the semi-classical Riemmanian Brownian bridge over a manifold with a pole.
\end{abstract}


 \section{Introdcution}
 Let $M$ be a finite dimensional smooth connected complete and stochastically complete Riemannian manifold $M$ whose Riemannian distance is denoted by $r$. 
 By stochastic completeness we mean that its minimal heat kernel satisfies that $\int p_t(x,y)dy=1$.  Denote by $C([0,1]; M)$ the space of continuous curves: $ \sigma: [0,1]\to M$, a Banach manifold modelled on the Wiener space. A chart containing a path $\sigma$ is given by a tubular neighbourhood of $\sigma$ and the coordinate map  is induced from the exponential map given by the Levi-Civita connection on the underlying finite dimensional manifold.
 For $x_0, y_0 \in M$ we denote by $C_{x_0}M$ and 
$C_{x_0, y_0}M$, respectively,  the based and the pinned space of continuous paths over $M$: \begin{equs}C_{x_0}M\;\;&=\{ \sigma \in C([0,1]; M): \quad \sigma(0)=x_0\}, \\
C_{x_0, y_0}M&=\{\sigma \in C([0,1]; M): \quad \sigma(0)=x_0, \quad  \sigma(1)=y_0\}.
\end{equs}
The pullback tangent bundle of $C_{x_0}M$ consisting of continuous $v: [0,1]\to TM$ with $v(0)=0$ and $v(t)\in T_{\sigma(t)}M$ where 
$\sigma \in C([0,1]; M)$ which for each $\sigma$ can be identified by parallel translation with continuous paths on $T_{x_0}M$, the latter is identified with $\R^n$ with a frame $u_0$. 
To define gradient operators we make a choice of a family of $L^2$ sub-spaces together with an Hilbert space structure,
and so we have a family of continuously embedded $L^2$ subspaces $\HH_\sigma$ and the $L^2$ sub-bundle $\HH:=\cup_\sigma \HH_\sigma$.
Firstly denote by $H$ the Cameron-Martin space over $\R^n$,    \begin{equation*}
H:= \bigg\lbrace h\in C( [0,1]; \R^n): h(0)=0, |h|_{H^1}: =\left(\int_0^1 |\dot h_s|^2 ds\right)^{\f 12} < \infty\bigg\rbrace,
\end{equation*}
with $H^0$ its subset consisting of  $h$ with $h(1)=0$.
If $\parals_\cdot(\sigma)$ denotes stochastic parallel translation along a path $\sigma$
we denote 
by $\HH_\sigma$ and $\HH^0_\sigma$ the Bismut tangent spaces:
\begin{equs}\HH_\sigma &=\{ \parals_\cdot(\sigma) h: h\in H\},  \qquad
\HH^0_\sigma =\{ \parals_\cdot(\sigma) h: h\in H,  h(1)=0\},
\end{equs} specifying
respectively the `admissible' tangent vectors at $\sigma\in C_{x_0}M$ and vectors at $\sigma\in C_{x_0, y_0}M$. 
These vector spaces are given the inner product inherited from the Cameron-Martin space $H$.

For an $L^2$ analysis on $C_{x_0, y_0}M$ we need a probability measure on it which is usually taken to be the probability distribution of the conditioned Brownian motion. The heat kernel measure, the distribution of a Brownian sheet, offers also an alternative measure, see \cite{Malliavin-loop, Driver-Lohrenz, Norris-sheets}.
See also \cite{Li-hypoelliptic-bridge} for a study of  the measure induced by a conditioned hypoelliptic stochastic process.
If we suppose that $M$ has a pole $y_0$, by which we mean that the exponential map $\exp_\pole: T_\pole M\to M$ is a diffeomorphism,  another probability measure, the probability distribution of the semi-classical Riemannian bridge, becomes available to us. For a simply connected Riemannian manifold with non-negative sectional curvature, every point is a pole.
 We denote this measure by $\nu=\nu_{x_0,\pole}$ and denote by $L^2(C_{x_0,y_0}M;\R)$ the corresponding $L^2$ space.

A semi-classical Riemannian Brownian bridge $(\tilde{x}_s, s\le 1)$
is a time dependent strong Markov process with generator $\frac{1}{2}\triangle + \nabla \log k_{1-s}(\cdot, y_0)$ where,   $$k_t(x_0, y_0):=(2\pi t)^{-\f n 2}e^{-\f {r^2(x_0, y_0)}{2t}}
J^{-\f 12}(x_0),$$
and $J(y)= | \det D_{\exp_\pole^{-1}(y)} \exp_\pole|$ is the Jacobian determinant of the exponential map at $y_0$.
Semi-classical Riemannian Brownian bridges (semi-classical bridge for short) were introduced by K. D. Elworthy and A. Truman  \cite{Elworthy-Truman-81}. 
For further explorations in this direction see  \cite{Elworthy-Truman-Waltling} and  \cite{Ndumn-09}. If $p_t$ is the heat kernel, the Brownian bridge is a Markov process with generator 
$\f 12 \Delta + \nabla \log p_{1-t}(\cdot, y_0)$.  Let us consider the two time dependent potential functions that drives the Brownian motion to the terminal value.  
They are close to each other as $t\to 1$, by Varadhan's asymptotic relations \cite{Varadhan}: $ (1-t) \log p_{1-t}(x,y_0)\sim-\f 1 2r^2(x,y_0)$.  There is also the relation  $\lim_{t\to 1} (1-t) \nabla \log p_{1-t}(x,y)=-\dot \gamma(0)$ where $\gamma$ is normal geodesic from $y_0$ to $x$. 
The two drift vector fields
$\nabla \log p_{1-t}(\cdot, y_0)$ and $\nabla \log k_{1-t}(\cdot, y_0)$ differ by $-\f 12 \nabla \log J$ near the terminal time.

Let us consider the unbounded linear differential operator $d$ on $L^2(C_{x_0,y_0}M;\R)$  taking values in $L^2(\cup_\sigma\HH_\sigma^*)$ where for $v\in \cup_\sigma\HH_\sigma^*$,
$$\|v_\cdot(\cdot)\|:=\left(\int_{C_{x_0,y_0}M }\left(|\parals_\cdot^{-1}v_\cdot(\sigma)|_H\right)^2  d\nu(\sigma)\right)^{\f 12}.$$
Another norm can be given, taking into accounts of the damping effects of the Riccic curvature, which will be discussed later. 
As the distance function from the semi-classical bridge to the pole is precisely the $n$-dimensional Bessel bridge where $n=\dim(M)$, the advantage of the semi-classical Brownian bridge measure is that it is easier to handle, which we demonstrate by studying the elementary property of the divergence operator.
Our main result is an integration by parts formula for $d$.   Such a formula is believed to be equivalent to an integration by parts formula. A proof for the equivalence was given in \cite{Elworthy-Li-ibp} for compact manifold and for the Brownian motion measure by induction. The same method should work here. However since it is a bridge measure the current method has its advantages. First order Feyman-Kac type formulas together with estimates for the gradient of the Feyman-Kac kernel using semi-classical bridge process and the damped stochastic parallel translation was obtained in  \cite{Li-Thompson}.

  Denote by $OM$  the space of orthonormal frames over $M$ and $\lbrace H_i \rbrace$ the canonical horizontal vector fields on $OM$ associated to an orthonormal basis $\{e_i\}$ of $\R^n$ so that $H_i$ is the horizontal lift of $ue_i$. For a tangent vector $v$ on $M$, we will denote by $\tilde v$ the horizontal lift of $v$ to $TOM$. Let  $\lbrace \Omega,\mathcal{F},\mathcal{F}_t,\mathbb{P}\rbrace$ be a filtered probability space on which is given a family of independent one-dimensional Brownian motions $\lbrace B^i \rbrace$. We define $B_t=(B_t^1, \dots, B_t^n)$. Let $u_0\in \pi^{-1}(x)$ be a frame at $x$, $u_t$ and $\tilde u_t$ be the solution to the stochastic differential equations, 
\begin{equation}\label{horizontal2}
d {u}_s = \sum_{i=1}^n  H_i({u}_s)\circ dB^i_s, \quad
d \tilde{u}_s = \sum_{i=1}^n  H_i(\tilde{u}_s)\circ dB^i_s + \tilde{A_s}(\tilde{u}_s)ds, \quad  \tilde u_0=u_0,
\end{equation}
where $\circ$ denote Stratonovich integration and $A_s=\nabla \log k_{1-s}(\cdot, y_0)$. Then $\tilde{x}_s := \pi(\tilde{u}_s)$ is a semi-classical Brownian bridge from $x_0$ to $y_0$ in time~$1$. 
 Let $\Ric_x$ denote the Ricci curvature at $x \in M$, by ${\Ric}^{\sharp}_x:T_xM \rightarrow T_xM$ we mean the linear map given by the relation $\< {\Ric}^\sharp_x u,v\> = {\Ric_x}(u,v)$.

Denote $r=r(\cdot, y_0)$ for simplicity. We will need the following geometric conditions.
Set \begin{equation}
\Phi=\frac{1}{2}J^{\frac{1}{2}}\triangle J^{-\frac{1}{2}}=\f 14|\nabla {\log} J|^2-\f 1 4 \Delta ({\log }J).
\end{equation}

\noindent
{\bf C1:} The Ricci curvature  is bounded.\\
{\bf C2:}  $|\nabla \Phi|+|\nabla( \log J)|\le c(e^{ a r^2}+1)$ for some $c>0$ and $a$ is an explicit constant to be given later.\\
{\bf C3:} $\Phi$ is bounded from below.\\
{\bf C4:} For each $t$, $k_t$ and $|\nabla k_t|$ are bounded,   $|\nabla \Phi|$ is bounded.\\
The condition that the Ricci curvature is bounded ensures that the solution to the canonical SDE is differentiable in the sense of Malliavin calculus. It also implies that $|\tilde W_t|$ is bounded and that the integration by parts formula holds for
the Brownian motion measure. Observe that $k_t$ and $|\nabla k_t|$  are bounded if $rJ^{-\f 12} $ and $J^{-\f 12} \nabla \log J^{-\f12}$ grow at most exponentially. Here we do not strive for the best possible conditions, as the optimal conditions will manifest themselves when  Clark-Ocone formula and  Poincar\'e inequalities are studied.


Our main results is the following integration by parts theorem.

  \begin{theorem}\label{thm1}
  Assume {\bf C1- C4} hold. 
   Then for any $F, G\in\Cyl$ and $h\in H^0$ the following integration by parts formula hold.
 \small
 \begin{equs}
&\int_{C_{x_0, y_0}M} G(\tilde x_\cdot)dF\left(\tilde u_\cdot(\sigma)  h_\cdot \right) \nu(d\sigma)+\int_{C_{x_0, y_0}M} F(\tilde x_\cdot)dG\left(\tilde u_\cdot(\sigma)  h_\cdot \right) \nu(d\sigma)\\
=&\E \left[(FG) (\tilde x_\cdot ) \int_0^1 \<\dot  h_s+\f 1 2 \tilde u_s^{-1}{\Ric}^{\sharp} 
(\tilde u_s h_s), d\tilde B_s\>\right]
+\E\left[(FG)(\tilde x_\cdot) \int_0^1 d\Phi(\tilde u_s h_s)ds \right].
\end{equs}
\normalsize
Here $d\tilde{B}_s = dB_s -\tilde u_s^{-1} \nabla \log k_{1-s}(\tilde{x}_s)\,ds$. In particular $d: \Cyl \subset  L^2(C_{x_0, y_0}M)\to L^2(\cup_\sigma \HH_\sigma)$ is closable, the domain of $d^*$ contains $\Cyl$ and
$$d^*G=-dG+G \int_0^1 \<\dot  h_s+\f 1 2 \tilde u_s^{-1}{\Ric}^{\sharp} 
(\tilde u_s h_s), d\tilde B_s\>+G\int_0^1 d\Phi(\tilde u_s h_s)ds.$$
  \end{theorem}
   For based path space over a compact manifold, with  Brownian motion measure (the Wiener measure), this was proved in \cite{Driver92},  for
non-compact manifolds see
\cite{Elworthy-Li-ibp, Elworthy-Li-icm}, 
\cite{Fang-Wang},  \cite{Thalmaier-Wang}, and 
\cite{Arnaudon-Driver-Thalmaier}. For  pinned manifolds with measure coming from the classical Brownian bridge measure, this was proved in \cite{Driver94}  and \cite{Hsu97}.

Let us now clarify the definition of $d$. A common definition for $d$, which we use,  is to take its initial domain to be $\Cyl$,  the set of cylindrical functions of the form
$F(\sigma)=f(\sigma_{t_1}, \dots, \sigma_{t_m})$ where $m\in \N$,
$0<t_1<t_2<\dots <t_m< 1$, and $f$ is a $BC^1$ function on the $m$-fold product space of $M$, or $\Cyl_0$ the subset containing $f(\sigma_{t_1}, \dots, \sigma_{t_m})$ where  $f$ is compactly supported. 
The $H$-derivative (Malliavin derivative) of $F$  in the direction of
$u_\cdot(\sigma) h_\cdot\in T_\sigma C_{x_0}M$~is:
$$(dF)(\parals_\cdot(\sigma) h_\cdot)=\sum_{k=1}^m \partial_{k}f \big(\parals_{t_k}(\sigma) h_{t_k}\big),$$
where $\partial_{k}f $ denotes the derivative of $f$ in its $k$th component and $\paral$ denotes parallel translation and identified with $u$ in the sequel. Denote by
$G(s,t)$ and $G^0(s,t)$, respectively, the Green's functions of $\f d {ds}$ on $(0,1)$ with suitable Dirichlet conditions: $G(s,t)=s\wedge t$ and $G^0(s,t)=s\wedge t -st$.
Then \begin{equs}(\nabla F)(\sigma)(t)&=\sum_{k=1}^m G(t_k,t) \parals_{t_k,t}(\sigma)\nabla_{k}f (\sigma_{t_1},\sigma_{t_2},\dots, \sigma_{t_m}),\\
(\nabla^0 F)(\sigma)(t)&=\sum_{k=1}^m G^0(t_k,t) \parals_{t_k,t}(\sigma)\nabla_{k}f (\sigma_{t_1},\sigma_{t_2},\dots, \sigma_{t_m}),
\end{equs}
where  $\nabla_k f$ denotes the gradient of $f$ in the $k$th variable.
 We have
\begin{equs}\|\nabla F\|^2&=\sum_{i,j=1}^m G(t_k, t_j) \<\paral_{t_k, t_j}\nabla_k f, \nabla_jf\>,\\ \quad \|\nabla^0 F\|^2&=\sum_{i,j=1}^m G^0(t_k, t_j) \<\paral_{t_k, t_j}\nabla_k f, \nabla_jf\>.
\end{equs}
It is an open problem whether the closure of $d$ with initial domain $BC^1$ agrees wth the closure of $d$ with initial value the cylindrical functions. This is the Markov uniqueness problem, this was studied In \cite{Elworthy-Li-CR}  where it was only proved that the latter including $BC^2$.

\section{Proof of Theorem \ref{thm1}}
\label{proof}

To clarify the singularities at the terminal time we first prove a lemma concerning the divergence of a suitable vector field on the path space.
Let $\tilde u_t$ be as defined by (\ref{horizontal2}), $\tilde x_t=\pi(\tilde u_t)$.  Recall that   $k_t(x_0, y_0)=(2\pi t)^{-\f n 2}e^{-\f {\rho^2(x_0, y_0)}{2t}}J^{-\f 12}(x_0)$
and $\tilde{B}_s = B_s - \tilde u_s^{-1}\nabla \log k_{1-s}(\tilde{x}_s, y_0)\,ds$. 
The reference to $y_0$ will be dropped from time to time for simplicity. Define $\ric_u=u^{-1}\Ric^\sharp u$.

\begin{lemma}\label{lemma1}
Assume stochastic completeness,  {\bf C2}, and $h\in H^{0}$. Then the following integral exists,
$$\int_0^1\left \<\dot h_s+\f 1 2 \ric_{\tilde u_s} (h_s), d\tilde B_s\right\>.$$
Furthermore, 
$$\lim_{t\to 1}\E\left( \left\<\nabla \ {\log} {k_{1-t}(\cdot)}, \tilde u_t h_t\right\>\right)^2=0,$$ 
$\int_0^t \<\dot h_s+\f 1 2 \ric_{\tilde u_s} (h_s),  d\tilde B_s\>$ converges, as $t\to 1$, in
$L^2(\Omega, \P)$; and
$$\begin{aligned}  \int_0^1\left \<\dot h_s+\f 1 2 \ric_{\tilde u_s} (h_s), d\tilde B_s\right\>
=&\int_0^1\left \<\dot h_s+\f 1 2 \ric_{\tilde u_s} (h_s), d B_s\right\>+\int_0^1 d\Phi (\tilde u_s h_s) ds\\
&+\int_0^1 \nabla d\left( { \log} k_{1-s} (\tilde x_s, y_0)\right) (\tilde u_sdB_s, \tilde u_sh_s).
\end{aligned}$$
\end{lemma}
\begin{proof}
The singularities in the integral $\int_0^1\left \<\dot h_s+\f 1 2 \ric_{\tilde u_s} (h_s), d\tilde B_s\right\>$ come from the involvement of $ \nabla \log k_{1-s}(\tilde{x}_s,y_0)$ and we only need to worry about
\begin{equation}
\label{singular}
\alpha_t:=\int_0^t \left\<\dot h_s+\f 1 2 \ric_{\tilde u_s} (h_s),\tilde u_s^{-1}\nabla \log k_{1-s}(\tilde{x}_s,y_0)\right\>ds.
\end{equation}
We  integrate by parts to deal with $\int_0^t \left\<\dot h_s,\tilde u_s^{-1}\nabla \log k_{1-s}(\tilde{x}_s,y_0)\right\>ds$, which involves the derivative of $h_s$.
 Since $\f D{ds} (u_sh_s)=u_s\dot h_s$, by stochastic calculus applied to 
$d \left( {\log} k_{1-s}\right)(u_s h_s)$, where $d$ denotes spatial differentiation with respect to the $M$-valued variable,
we see that 
\begin{equs}
& \left\< \nabla {\log} {k_{1-t}(\tilde x_t)}, \tilde u_t h_t\right\>\\
=&\int_0^t  \left\<\nabla { \log} k_{1-s} (\tilde x_s), \tilde u_s\dot h_s \right\>\, ds+\sum_{i=1}^n\int_0^t \nabla d\left( { \log} k_{1-s}\right) (\tilde u_se_i, \tilde u_sh_s) dB_s^i\\
&+\int_0^t \nabla d\left( { \log} k_{1-s})  \left(\nabla {\log} k_{1-s} (\tilde x_s), \tilde u_sh_s \right)\right)\, ds\\
&+\int_0^t\left( \f 1 2\trace \nabla^2 + {\partial_r} \right)\left( D \left( {\log} {k_{1-s}(\tilde x_s)}\right)\right)(\tilde u_s h_s)\, ds,
\end{equs}
the first term on the right hand side being $\alpha_t$.
Since  $\nabla  {\log} {k_{1-s}(x)}=-\f {r (x) \nabla r(x)}{1-s} +\nabla {\log} (J^{-\f 12})$, $\Delta r={n-1 \over r} +\<\nabla r, \nabla \log J\>$,  the following set of formulas are easy to verify.
\begin{equation}
\label{formulas}
{\begin{split}
\Delta {\log} k_{1-s}&=-\f n {1-s} -\f {r\<\nabla r,\nabla \log J\>}{1-s}-\f 12 \Delta(\log J),\\
{\partial_ r} \log k_{1-s} &=\f n  {2(1-s)}-\f {r^2} {2(1-s)^2},\\
|\nabla \log k_{1-s}|^2&=\f {r^2}{(1-s)^2}+\f 12|\nabla \log J|^2+\f {r\<\nabla r,\nabla \log J\>}{1-s}.
\end{split}}
\end{equation}
It follows that 
$$\left(\f 1 2 \Delta + {\partial_r}\right)\left( {\log} {k_{1-s}}\right) +\f 1 2|\nabla \log k_{1-s} |^2=\f 14|\nabla {\log} J|^2-\f 1 4 \Delta ({\log }J)=\Phi.$$
Let $\Delta^1:=(dd^*+d^*d)$ denote the Laplace-Beltrami Kodaira operator on differential 1-forms. 
By the Weitzenb\"ock formula,
$\left( \f 1 2\trace \nabla^2 + {\partial_r} \right)d =\left( \f 1 2\Delta^1 d +\f 1 2 \Ric^{\sharp}d+ {\partial_r}d \right),$
 and consequently,
\begin{equs}
&\left( \f 1 2\trace \nabla^2 + {\partial_r} \right)d \left( {\log} {k_{1-s}(\tilde x_s)}\right)\\
=&d\left( \f 1 2 \Delta + {\partial_r} \right) \left( {\log} {k_{1-s}(\tilde x_s)}\right)
+\f 1 2 \Ric^{\sharp}\left(d {\log} {k_{1-s}(\tilde x_s)}\right)\\
=&-\f 1 2 d(|\nabla {\log k}_{1-s}(\cdot) |^2)+d\Phi
+\f 1 2 \Ric^{\sharp}\left(d {\log} {k_{1-s}(\tilde x_s)}\right).
 \end{equs}
 Thus
 \begin{equs}{}& \nabla d \left( { \log} k_{1-s})  \left(\nabla {\log} k_{1-s} (\tilde x_s), \tilde u_sh_s \right)\right)
  +\left( \f 1 2\trace \nabla^2 +\f {\partial} {\partial r} \right)\left( d \left( {\log} {k_{1-s}}\right)\right)(\tilde u_s h_s)\\
  &=d\Phi( \tilde u_sh_s)
+\f 1 2 \Ric\left(\nabla{\log} {k_{1-s}(\tilde x_s)}, \tilde u_sh_s\right).
\end{equs}
Let us return to $ \left\< \nabla {\log} {k_{1-t}(\tilde x_t)}, \tilde u_t h_t\right\>$:
\begin{equs}
& \left\< \nabla {\log} {k_{1-t}(\tilde x_t)}, \tilde u_t h_t\right\>\\
=&\int_0^t  \left\<\nabla { \log} k_{1-s} (\tilde x_s), \tilde u_s\dot h_s \right\>\, ds
+\sum_{i=1}^n\int_0^t \nabla d\left( { \log} k_{1-s} \right) (\tilde u_se_i, \tilde u_sh_s) dB_s^i\\
&+\int_0^t d\Phi( \tilde u_sh_s)\;ds
+\f 1 2 \int_0^t\Ric\left(\nabla{\log} {k_{1-s}(\tilde x_s)}, \tilde u_sh_s\right)\;ds.
\end{equs}
We thus obtain
 the following relation:
\begin{equs}
\alpha_t=&\left\< \nabla {\log} {k_{1-t}(\tilde x_t)}, \tilde u_t h_t\right\>+
 \f 1 2\int_0^t D \log k_{1-s}(\Ric^{\sharp}(\tilde u_s h_s))ds\\=&
\left\<\nabla \ {\log} {k_{1-t}(\cdot)}, \tilde u_t h_t\right\>-\int_0^t \left\< \nabla \Phi, \tilde u_sh_s\right\> ds\\
&-\sum_{i=1}^n\int_0^t \nabla d\left( { \log} k_{1-s} \right) (\tilde u_se_i, \tilde u_sh_s) \,dB_s^i.
\end{equs}
We will prove that each of the terms on the right hand side converges as $t$ approaches $1$. Furthermore
$\left\<\nabla \ {\log} {k_{1-t}(\cdot)}, \tilde u_t h_t\right\>$ converges to zero.
 We first observe that there exists a constant $C$ such that $\E[ r(\tilde x_t)^p]
\le Ct^{\f p2}$.
Indeed  $r_t:=\rho(\tilde x_t, y_0)$ satisfies
\begin{equs}
r_t-r_0=&
\beta_t +\int_0^t{1\over 2} \Delta r(\tilde{x}_s) ds-
\int_0^t  {r(\tilde{x}_s)  \over 1-s}ds
 -{1\over 2} \int_0^t \left\< \nabla r,  \nabla \log J \right\>_{\tilde{x}_s}ds\\
 =&\beta_t +\int_0^t \f {n-1} {2r_s} ds-
\int_0^t  \f{r_s}  {1-s}ds,
\end{equs}
where $\beta_t$ is a one dimensional Brownian motion and we have used the fact that
$\Delta r={n-1 \over r} +\<\nabla r, \nabla \log J\>$. Thus $r_s$ is a Bessel bridge starting at $\rho(x_0, y_0)$ and ending at $0$ at time $\term$. In particular $\lim_{t\uparrow 1} \tilde x_t=y_0$ and $(r_t, t\le 1)$  is a continuous semi-martingale. Furthermore for any $p>1$, $\E [r(\tilde x_t)^p] \le C t^{\f p 2}$. If $K_t$ denotes the standard Gaussian kernel on $\R^n$ then for $z_1, z_2\in \R^n$ with $|z_1-z_2|=\rho(x_0, y_0)$,
$$
\E[ r(\tilde x_t, y_0)^p]=\f 1 {K_T(z_1, z_2)}\int_{\R^n} |z-z_2|^p K_t(z_1, z) K_{1-s}(z, z_2) dz\le C |z-z_1|^{\f p 2}.$$ 
 We also know that 
$\E [e^{2a r_t^2}]<\infty$ for some $a$ and $t\le 1$, involking condition C2.

We show below that (\ref{singular}) has a limit as $t\to 1$.
Firstly, since $|d\Phi|\le c e^{ar^2}$,  
$$\lim_{t\to 1}\E \left[\int_t^1 \left\< \nabla \Phi, \tilde u_sh_s\right\> ds\right]^2= 0.$$ 
 We work with the first term on the right hand side: $$\left\<\nabla \ {\log} {k_{1-t}(\cdot, y_0)}, \tilde u_t h_t\right\>=\f { r(\tilde x_t) \<\nabla r(\tilde x_t), \tilde u_t h_t\>}{1-t} +
\<\nabla {\log} J_{\tilde x_t}^{-\f 12},\tilde u_t h_t\>.$$
 Since  $|d( {\log} J_x^{-\f 12})|\le c e^{ar(x)^2}$, 
$\lim_{t\to 1}\<\nabla {\log} J_{\tilde x_t}^{-\f 12},\tilde u_t h_t\>$ converges  in $L^2(\Omega)$. Thus \begin{equation}\label{2.1}
\lim_{t\uparrow 1} \E \left|\<\nabla {\log} J^{-\f 12}(\tilde x_t),\tilde u_t h_t\>\right|^2=0,
\end{equation}
using the fact that $h_t\to 1$.
Also, by the symmetry of the Euclidean bridge,  $\E[ r^2(\tilde x_t, y_0)]\le C\left( t\wedge (1-t)\right)$ and hence
$$\E \left |\f {r (\tilde x_t) \left\<\nabla r(\tilde x_t),\tilde u_t h_t\right\>}{1-t} \right|^2\le  \f {|h_t|^2} {1-t}.$$
Since $h_1=0$, and $h\in H$,
$$\f{ |h_t|^2} {1-t}=\f 1 {1-t}\left| \int_t^1 \dot h_s dr\right|^2\le \int_t^1 |\dot h_s|^2 ds\to 0,$$
as $t\to 1$, using the fact that $h\in H$.
We conclude that 
$$\lim_{t\to 1}\E \left[\left\<\nabla \ {\log} {k_{1-t}(\cdot)}, \tilde u_t h_t\right\>\right]^2=0.$$
 For the final term we observe that
$$\nabla d\left( { \log} k_{1-s} \right) (\tilde u_se_i, \tilde u_sh_s)
=-\f {\nabla r(\tilde  u_se_i)\nabla r(\tilde u_s h_s)}{1-s}-\f {r \nabla dr (\tilde u_se_i, \tilde u_sh_s)}{1-s}.$$
We further observe that the Frobenius norm of
the Hessian of the distance function satisfies:
$$\|\nabla d r\|_F:=\left(\sum_{i,j=1}^n \<\nabla_{E_i} {\partial r}, E_j\>\right)^{\f 12} \le \f 1{\sqrt{n-1}}\Delta r
\le   \f 1{\sqrt{n-1}} \left({n-1 \over r} +\<\nabla r, \nabla \log J\>\right).$$
Since $|\nabla \log J|\le ce^{a r^2}$,  for some constant $C$, which may depend on $n$,
$${\begin{split}
&\E \left[\sum_{i=1}^n\int_0^t \nabla d\left( { \log} k_{1-s} \right) (\tilde u_se_i, \tilde u_sh_s) \,dB_s^i\right]^2\\
&\le  C\int _0^t \f {|h_s|^2} {(1-s)^2}ds\le C\f { |h_t|^2} {1-t}+4C\int_0^t |\dot h_s|^2 ds.
\end{split}}$$
This follows from the following standard computation, $$\int_0^t  \f {|h_s|^2} {(1-s)^2}ds
=   \f {|h_t|^2} {1-t}-\int_0^t  \f {\<h_s, 2\dot h_s\>} {(1-s)^2}ds
\le \f { |h_t|^2} {1-t}+\f 1 2 \int_0^t |h_s|^2 ds+2\int_0^t |\dot h_s|^2 ds.$$
This concludes the proof of the convergence of the integral. The required identity follows from the formula,
 given earlier, for $\alpha_t$.
\end{proof}

 Let $u_t$ be the solution to the equation 
 $du_t = \sum_{i=1}^n H_i(u_t) \circ dB^i_t$ with initial value $u_0 \in \pi^{-1}(x_0)$. Then $x_t := \pi (u_t)$ is a Brownian motion on $M$ starting at $x_0$ and
 the integration by parts formula holds on $L^2(C_{x_0}M; \mu)$.  
  For any $F,G\in \Cyl$,  and $h\in H(T_{x_0}M)$ with $h(0)=0$, $d$  is the differential on $L^2(C_{x_0}M)$ with respect to the Brownian motion measure:
\begin{equation}
\label{ibp}
\E [dF(u_\cdot h_\cdot) G]=-\E [FdG(u_\cdot h_\cdot)]
+\E \left[FG\int_0^1 \<\dot h_s+\f 1 2{u_s^{-1} {\Ric}^{\sharp}(u_sh_s)}, dB_s\>\right].
\end{equation}
If $M$ is compact, see e.g. B. Driver \cite{Driver92}. This is also known to hold if the Ricci curvature is bounded from below.
The divergence of $u_\cdot h_\cdot$ is
$$\div(u_\cdot h_\cdot)=\int_0^1 \<\dot h_s+\f 1 2{u_s^{-1} {\Ric}^{\sharp}_{u_s}(u_sh_s)}, dB_s\>.$$

The following lemma completes the proof of Theorem \ref{thm1}.
\begin{lemma}
Suppose stochastic completeness,  {\bf C2-C4}, and suppose that the integration by parts formula (\ref{ibp}) holds for the Brownian motion measure. Then the conclusion of Theorem \ref{thm1} holds.
\end{lemma}

Let $h\in H^{0}$.
Our plan is to pass the integration on the path space to the pinned path space
by a Girsanov transform. 
We first observe that  if $F\in \Dom(d)$, adapted to $\G_t$ where $t<1$, then
$$\E[ dF(\tilde u h_\cdot)]=\E \left[dF ( u h_\cdot)\f {k_{1-t}(x_t)}{k_1(x_0)}e^{-\int_0^t \Phi(x_s)ds}\right].$$
In fact, the formula for the probability density between the original probability measure, on $\G_t$,  and the one for which
$B_t-\int_0^t \<u_sdB_s, \nabla {\log} k_{1-s}(x_s)\>$ is a Brownian motion, is:
\begin{equs}
M_t=\exp \left[  \sum_{i=1}^m\int_0^\tm  \<\nabla {\log} k_{\term-\stm}  ( x_\stm, y_0),  u_\stm e_i\> dB_\stm^i -{1\over 2} \int_0^\tm | \nabla {\log} k_{\term-\stm}(x_\stm, y_0) |^2 d\stm\right].
\end{equs}
By an application of It\^o's formula, and identities (\ref{formulas}) in the proof of Lemma \ref{lemma1},
$$M_t=\f {k_{1-t}(x_t, y_0)} {k_1(x_0, y_0)}  \exp  \left(-\int_0^t \Phi(x_s)ds\right). $$
Since the Brownian motion and the semi-classical bridge are conservative, then $(M_s, s\le t)$ is a martingale for any $t<1$.

Since $\Phi$ is bounded from below and has bounded derivative, 
$e^{-\int_0^t \Phi(\tilde x_s) ds}$ can be approximated by smooth cylindrical functions in the domain of $d$.
Next we observe that
$$\nabla {k_{1-s}(\cdot, y_0)}= {2\pi (1-s)^{-\f n 2}}e^{-\f {r^2}{2(1-s)}}J^{-\f 12}
\left(-\f {r \nabla r }{1-s} +\nabla {\log} J^{-\f 12}\right),$$
is bounded and smooth, so $\f {k_{1-t}(x_t, y_0)}{k_1(x_0, y_0)}e^{-\int_0^t \Phi(x_s)ds}$ belongs to the domain of $d$. 
Consequently, for $F, G$ measurable with respect to the canonical filtration up to time $t<1$, we apply (\ref{ibp}) to see \begin{equs}
&\E [GdF(\tilde u_\cdot h_\cdot] =\E \left [dF ( u _\cdot h_\cdot) G(\tilde x_\cdot)M_t\right]\\
&=\E \left[(FG)( x_\cdot ) M_t\div(u_\cdot h_\cdot)\right]
-\E \left[(FG) ( x_\cdot ) dM_t(u_\cdot h_\cdot)\right]-\E [F(x_\cdot) dG( uh) M_t]\\
=&\E \left[F(  x_\cdot )M_t\int_0^t \<\dot  h_s+\f 1 2{ u_s^{-1} {\Ric}^{\sharp}_{u_s}(h_s)}, dB_s\>\right]-\E [F(\tilde x_\cdot) dG(\tilde uh) ]\\
&-\E \left[F ( x_\cdot ) M_td\left({\log} {k_{1-t}(x_t, y_0)}-\int_0^t \Phi(x_s)ds\right)(u_\cdot h_\cdot)\right]\\
=&\E \left[F (\tilde x_\cdot ) \int_0^t \<\dot  h_s+\f 1 2{\tilde u_s^{-1} {\Ric}^{\sharp}_{\tilde u_s}(h_s)}, d\tilde B_s\>\right]-\E [F(\tilde x_\cdot) dG(\tilde uh) ]\\
&-\E \left[F(\tilde x_\cdot)\< \nabla {\log} k_{1-t} (\tilde x_t, y_0),\tilde u_t h_t\> -F(\tilde x_\cdot) \int_0^t d \Phi(u_s h_s)ds\right].
\end{equs}
We take $t\uparrow 1$, by (\ref{2.1}) and Lemma \ref{lemma1}, $\lim_{t\uparrow 1} \<\nabla {\log} k_{1-t} (\tilde x_t, y_0),\tilde u_t h_t\>=0$ in $L^2$,
\begin{equs}
&\E [GdF(\tilde u_\cdot h_\cdot] +\E [F(\tilde x_\cdot) dG(\tilde uh) ]\\
=&\E \left[(FG) (\tilde x) \int_0^1 \<\dot h_s+\f 1 2{\tilde u_s^{-1} {\Ric}^{\sharp}(\tilde u_s h_s)}, d\tilde B_s\>\right]
+ \E\left[(FG)(\tilde x)  \left( \int_0^1 \;d \Phi(\tilde u_s h_s)ds\right) \right].
\end{equs}
In particular, $\Dom(d^*)\supset \Cyl$,
$$d^*1=\int_0^1 \<\dot h_s+\f 1 2{\tilde u_s^{-1} {\Ric}^{\sharp}(\tilde u_s h_s)}, d\tilde B_s\>+ \left( \int_0^1 \;d \Phi(\tilde u_s h_s)ds\right),$$
and $d^*$ is a closable operator. This completes the proof of the Lemma.

\subsection{Comment}

Let us consider briefly for which manifolds our assumptions on $\Phi$ hold.
Denote by $\partial r$ the radial curvature which,  evaluated at $x\in M$, is the unit vector field tangent to the normal geodesic between $x$ and the pole pointing away from the pole.  The Hessian of $r$ describes the change of the Riemannian tensor in the radial directions,
 while the change of the volume form in the radial direction is associated to the Laplacian of $r$. More precisely we have:
 $$L_{\partial r}g=2 \Hess (r), \qquad L_{\partial r} d \text{vol}=\Delta r d \text{vol}, \qquad \Delta r=\f {n-1} r+d r (\nabla \log J),$$ 
indicating how the Jacobian determinant adjusts the speed of the convergence so that the semi-classical bridge behaves exactly like 
the Euclidean Brownian bridge.

  For the Hyperbolic space,  $\Phi$ is bounded from the formula below,
$\Phi=-\f 1 8 (n-1)^2c^2+\f 1 8(n-1)(n-3) \left(\f 1 {r^2}- c^2 \sinh^{-2} (rc)\right)$.
If $(N, o)$ is a model space, 
 its Riemannian metric in the geodesic polar coordinates takes the form $g=dr^2+f^2(r) d\theta^2$, then 
on $N\setminus\{o\}$, $\Hess (r)=\f {f'(r)}{f(r)}(g -dr\otimes dr)$. 
For the hyperbolic space of constant sectional
curvature $-c^2$, the Riemannian metric is $g=dr^2+(\f 1 c \sinh (cr))^2 d\theta^2$. Also $\Hess (r^2)=2 dr\otimes dr +2cr \coth (cr) (g -dr\otimes dr)$. Furthermore its Jacobian determinant 
is  $J= (\f {\sinh (cr)}{c r})^{(n-1)}$.

For manifolds of non-constant curvature we may use the Hessian comparison theorem.
The radial curvature at a point $x\in M$ is the sectional curvature in a plane at $T_xM$ containing the radial vector field $\partial_r$.   Let us recall a comparison theorem from \cite[R. E. Greene and H. Wu]{Greene-Wu}:
let $(N, o)$ be another Riemannian manifolds with a pole which we denote by $o$. Suppose that $(\gamma(t), t\in [0, b])$ is a normal geodesic  in $M$ 
with the initial point $\pole$ and $(\gamma_2(t): t\in [0, b])$ a normal geodesic in $N$ from $o$.  
We suppose that the radial curvature at $\gamma_2(t)$ is greater than or equal to the radial curvatures at $\gamma(t)$. 
By this we mean the curvature operator $\Rc$ on $M$ and $\Rc_2$ on $N$ satisfy the relation
$\left\<\Rc(w, \dot \gamma)w, \dot \gamma\right\>\le \left\<\Rc_2(w_2, \dot \gamma_2)w_2, \dot \gamma_2\right\>$
for any unit vectors  $w\in ST_{\gamma(t)}M$ and $w_2\in ST_{\gamma_2(t)}N$,
satisfying the relation $\<w, \partial_r\>=\<w_2,  \partial_r\>$ where $\partial_r$ denotes the radial vector fields for both manifolds.
Then for any nondecreasing function $\alpha: \R_+\to \R$, $\Hess (\alpha\circ r_2)(\gamma_2(t))\le \Hess (\alpha\circ r)(\gamma(t))$,
where $r_2$ is the Riemannian distance function on $N$ from $o$.

\section{Conclusion}
We have proved an integration by parts formula on $L^2(C_{x_0, y_0}, \nu)$ where
$\nu$ is the probability measure induced by the semi-classical bridge. 
A probability measure $\mu$ on the path space is said to satisfy the Poincar\'e inequality if there exists a constant $c$ such that
$$\int \left(F- \int F d\mu \right)^2 d\mu \le c\int \left(|\nabla F|_{\HH}\right)^2d\mu $$
for all $F\in \dom(d)$ and the inner product on $\HH$ can be defined either by stochastic parallel translation or by damped stochastic parallel translation. 

{\bf Conjecture.}  A Poincar\'e inequality holds for the semi-classical bridge measure on
a class of Cartan-Hadamard manifolds. Of course it is reasonable to assume  growth conditions on $J$, $J^{-1}$ and suitable conditions on the range of the sectional curvature. 

We remark that, for the Brownian bridge measure the question whether the Poincar\'e inequality holds is not solved satisfactorily. The spectral gap inequality is known to hold for Gaussian measure on $\R^n$ by L. Gross \cite{Gross}, who also 
made a conjecture on its validity.  The spectral gal inequality has been proven to hold on the hyperbolic space  \cite{CLW},
see also \cite{Aida,Gong-Ma,Airault-Malliavin,Fang99,Elworthy-LeJan-Li-book}. 
A counter example exists \cite{Eberle}, see also the more recent articles \cite{Hino, Gong-Rockner-Wu}.

\def\cprime{$'$} \def\cprime{$'$}


\begin{thebibliography}{10}

\bibitem{Aida}
Shigeki Aida.
\newblock Logarithmic derivatives of heat kernels and logarithmic {S}obolev
  inequalities with unbounded diffusion coefficients on loop spaces.
\newblock {\em J. Funct. Anal.}, 174(2):430--477, 2000.

\bibitem{Airault-Malliavin}
H.~Airault and P.~Malliavin.
\newblock Integration on loop groups. {II}. {H}eat equation for the {W}iener
  measure.
\newblock {\em J. Funct. Anal.}, 104(1):71--109, 1992.

\bibitem{Arnaudon-Driver-Thalmaier}
Marc Arnaudon, Bruce~K. Driver, and Anton Thalmaier.
\newblock Gradient estimates for positive harmonic functions by stochastic
  analysis.
\newblock {\em Stochastic Process. Appl.}, 117(2):202--220, 2007.

\bibitem{CLW}
Xin Chen, Xue-Mei Li, and Bo~Wu.
\newblock A {P}oincar\'e inequality on loop spaces.
\newblock {\em J. Funct. Anal.}, 259(6):1421--1442, 2010.

\bibitem{Driver92}
Bruce~K. Driver.
\newblock A {C}ameron-{M}artin type quasi-invariance theorem for {B}rownian
  motion on a compact {R}iemannian manifold.
\newblock {\em J. Funct. Anal.}, 110(2):272--376, 1992.

\bibitem{Driver94}
Bruce~K. Driver.
\newblock A {C}ameron-{M}artin type quasi-invariance theorem for pinned
  {B}rownian motion on a compact {R}iemannian manifold.
\newblock {\em Trans. Amer. Math. Soc.}, 342(1):375--395, 1994.

\bibitem{Driver-Lohrenz}
Bruce~K. Driver and Terry Lohrenz.
\newblock Logarithmic {S}obolev inequalities for pinned loop groups.
\newblock {\em J. Funct. Anal.}, 140(2):381--448, 1996.

\bibitem{Eberle}
Andreas Eberle.
\newblock Spectral gaps on loop spaces: a counterexample.
\newblock {\em C. R. Acad. Sci. Paris S\'er. I Math.}, 330(3):237--242, 2000.

\bibitem{Elworthy-Truman-81}
David Elworthy and Aubrey Truman.
\newblock Classical mechanics, the diffusion (heat) equation and the
  {S}chr\"odinger equation on a {R}iemannian manifold.
\newblock {\em J. Math. Phys.}, 22(10):2144--2166, 1981.

\bibitem{Elworthy-Truman-Waltling}
David Elworthy, Aubrey Truman, and Keith Watling.
\newblock The semiclassical expansion for a charged particle on a curved space
  background.
\newblock {\em J. Math. Phys.}, 26(5):984--990, 1985.

\bibitem{Elworthy-LeJan-Li-book}
K.~D. Elworthy, Y.~Le~Jan, and Xue-Mei Li.
\newblock {\em On the geometry of diffusion operators and stochastic flows},
  volume 1720 of {\em Lecture Notes in Mathematics}.
\newblock Springer-Verlag, Berlin, 1999.

\bibitem{Elworthy-Li-ibp}
K.~D. Elworthy and Xue-Mei Li.
\newblock A class of integration by parts formulae in stochastic analysis. {I}.
\newblock In {\em It\^o's stochastic calculus and probability theory}, pages
  15--30. Springer, Tokyo, 1996.

\bibitem{Elworthy-Li-CR}
K.~David Elworthy and Xue-Mei Li.
\newblock Gross-{S}obolev spaces on path manifolds: uniqueness and intertwining
  by {I}t\^o maps.
\newblock {\em C. R. Math. Acad. Sci. Paris}, 337(11):741--744, 2003.

\bibitem{Elworthy-Li-icm}
K.~David Elworthy and Xue-Mei Li.
\newblock Geometric stochastic analysis on path spaces.
\newblock In {\em International {C}ongress of {M}athematicians. {V}ol. {III}},
  pages 575--594. Eur. Math. Soc., Z\"urich, 2006.

\bibitem{Fang99}
Shizan Fang.
\newblock Integration by parts formula and logarithmic {S}obolev inequality on
  the path space over loop groups.
\newblock {\em Ann. Probab.}, 27(2):664--683, 1999.

\bibitem{Fang-Wang}
Shizan Fang and Feng-Yu Wang.
\newblock Analysis on free {R}iemannian path spaces.
\newblock {\em Bull. Sci. Math.}, 129(4):339--355, 2005.

\bibitem{Gong-Ma}
Fu-Zhou Gong and Zhi-Ming Ma.
\newblock The log-{S}obolev inequality on loop space over a compact
  {R}iemannian manifold.
\newblock {\em J. Funct. Anal.}, 157(2):599--623, 1998.

\bibitem{Gong-Rockner-Wu}
Fuzhou Gong, Michael R{\"o}ckner, and Liming Wu.
\newblock Poincar\'e inequality for weighted first order {S}obolev spaces on
  loop spaces.
\newblock {\em J. Funct. Anal.}, 185(2):527--563, 2001.

\bibitem{Greene-Wu}
R.~E. Greene and H.~Wu.
\newblock {\em Function theory on manifolds which possess a pole}, volume 699
  of {\em Lecture Notes in Mathematics}.
\newblock Springer, Berlin, 1979.

\bibitem{Gross}
Leonard Gross.
\newblock Logarithmic {S}obolev inequalities.
\newblock {\em Amer. J. Math.}, 97(4):1061--1083, 1975.

\bibitem{Hino}
Masanori Hino.
\newblock Exponential decay of positivity preserving semigroups on {$L^p$}.
\newblock {\em Osaka J. Math.}, 37(3):603--624, 2000.

\bibitem{Hsu97}
Elton~P. Hsu.
\newblock Integration by parts in loop spaces.
\newblock {\em Math. Ann.}, 309(2):331--339, 1997.

\bibitem{Li-hypoelliptic-bridge}
Xue-Mei Li.
\newblock On hypoelliptic bridge.
\newblock {\em Electron. Commun. Probab.}, 21:Paper No. 24, 12, 2016.

\bibitem{Li-Thompson}
Xue-Mei Li and J.~Thompson.
\newblock First order feynman-kac formula.
\newblock In preparation, 2016.

\bibitem{Malliavin-loop}
Paul Malliavin.
\newblock Hypoellipticity in infinite dimensions.
\newblock In {\em Diffusion processes and related problems in analysis, {V}ol.\
  {I} ({E}vanston, {IL}, 1989)}, volume~22 of {\em Progr. Probab.}, pages
  17--31. Birkh\"auser Boston, Boston, MA, 1990.

\bibitem{Ndumn-09}
Martin~N. Ndumu.
\newblock Heat kernel expansions in vector bundles.
\newblock {\em Nonlinear Anal.}, 71(12):e445--e473, 2009.

\bibitem{Norris-sheets}
J.~R. Norris.
\newblock Twisted sheets.
\newblock {\em J. Funct. Anal.}, 132(2):273--334, 1995.

\bibitem{Thalmaier-Wang}
Anton Thalmaier and Feng-Yu Wang.
\newblock Gradient estimates for harmonic functions on regular domains in
  {R}iemannian manifolds.
\newblock {\em J. Funct. Anal.}, 155(1):109--124, 1998.

\bibitem{Varadhan}
S.~R.~S. Varadhan.
\newblock Diffusion processes in a small time interval.
\newblock {\em Comm. Pure Appl. Math.}, 20:659--685, 1967.

\end{thebibliography}
\end{document}